\documentclass[letterpaper,10pt,onecolumn,final]{amsart}
\usepackage{amsmath,amsfonts,amssymb}
\usepackage[bookmarksnumbered, colorlinks, plainpages]{hyperref}
\hypersetup{colorlinks=true,linkcolor=red, anchorcolor=green, citecolor=cyan, urlcolor=red, filecolor=magenta, pdftoolbar=true}

\newtheorem{theorem}{Theorem}[section]

\newtheorem{proposition}[theorem]{Proposition}
\newtheorem{corollary}[theorem]{Corollary}
\theoremstyle{definition}
\newtheorem{definition}[theorem]{Definition}
\newtheorem{example}[theorem]{Example}

\theoremstyle{remark}
\newtheorem{remark}[theorem]{Remark}
\numberwithin{equation}{section}
\newcommand{\R}{\mathbb{R}}
\makeatletter
\@namedef{subjclassname@2020}{%
	\textup{2020} Mathematics Subject Classification}
\makeatother

\begin{document}
\setcounter{page}{1}

\title[Harmonically $m$-convex set-valued function]{Harmonically $m$-convex set-valued function}

\author[Gabriel Santana, Maira Valera-L\'opez]{Gabriel Santana$^1$, Maira Valera-L\'opez$^{1*}$}
\address{$^{1}$ Escuela de Matem\'atica,  Facultad de Ciencias, Universidad Central de Venezuela, Caracas 1010, Venezuela.}
\email{\textcolor[rgb]{0.00,0.00,0.84}{gaszsantana@gmail.com}}
\email{\textcolor[rgb]{0.00,0.00,0.84}{maira.valera@ciens.ucv.ve}}


\date{\newline \indent $^{*}$ Corresponding author: Gabriel Santana}

\begin{abstract}
	This research aimed to introduce the concept of harmonically $m$-convex  set-valued functions, which is obtained from the combination of two definitions: harmonically $m$-convex functions and set-valued functions. In this work some properties and characteristics are developed, as well as a inequality of the Hermite-Hadamard type for such functions.
\end{abstract} \maketitle

\section{Introduction}

Convexity is a fundamental concept in mathematics and has interesting applications in other areas of knowledge. The study of this definition has increased exponentially at the beginning of this century obtaining important results in this area (see \cite{Merentes2013}). 
In recent studies, new concepts of convexity have been introduced from the original definition, one of them is harmonic convexity. These concepts constitute the basis for the development of our work. Anderson \textit{et. al.} \cite{Anderson2004} and I. Iscan \cite{Iscan2014} have considered harmonic convex functions which can be viewed as an important class of convex functions. One can show that harmonic convex functions have some nice properties, which convex functions have. Noor \textit{et. al.} \cite{Noor2015} introduced the class of harmonic $h$-convex functions with respect to an arbitrary nonnegative function $h(\cdot)$. This class is more general and contains several known and new classes of harmonic convex functions as special cases.

Several new convex sets have been introduced and investigated (see \cite{Aumann1965, Narvaes2011}).  However, the concept of $m$-convex set $D$ was introduced in 1984 by G. Toader \cite{Toader1984}  and go on to presenting the definition of $m$-convex set value function and stating some properties and examples. 

The notion of set-valued function arises at the beginning of the twentieth century, when C. Berge \cite{Berge1963} introduced the concept of upper and lower limit of succession sets. Recently different researchers have been studied for different notions of set-valued convexity functions, as well as have been used to find the error in some inclusion problems with set restrictions, (see \cite{Geletu2006, Leiva2013, Matkowski1998, Mejia2014, Nikodem1987, Nikodem1989, Nikodem2014, Sadowska1996, Santana2018} ). In particular, T. Lara \textit{et al.} \cite{Lara2019} introduced the notion of $m$-convex set-valued functions defined on a nonempty $m$-convex subset of a real linear space with values in the set of nonempty parts of an equality real linear space. In this research are given several characterizations as well as certain algebraic properties and examples. 

Various integral inequalities for convex functions and their variant forms are being developed using novel ideas and techniques. Some recent developments demonstrate results of Hermite-Hadamard inequality for harmonically convex set-valued functions. Santana \textit{et al.} \cite{Santana2018} introduced the definition of harmonically convex and strongly harmonically convex set-valued functions, obtaining important results such as Hermite-Hadamard and Fej\'er inequalities, as well as a Benstein-Doetsch-type theorem.

The aim of this research is to introduce a new concept of harmonically $m$-convex set-valued function, in turn, this research explores some characteristics and properties that exhibit these. Also, a new Hermite-Hadamard type inequality for harmonically $m$-convex set-valued functions is developed. The results obtained in this investigation are being studied for stochastic processes, having as reference the following works \cite{Barraez2015, Gonzalez2015, Gonzalez2016, Materano2015, Materano2016, Materano2017}.

\section{Preliminary}

As part of our research it is necessary to provide the reader with some preliminary definitions used throughout this investigation in order to lay the foundations for the development of this work.

\begin{definition}(see \cite{Iscan2016})
	Let $X$ a linear space and $m\in (0,1]$. A nonempty subset $D$ of $X$ is said to be harmonic $m$-convex, if for all $x,y\in D$ and $t\in [0,1]$, we have:
	
	$$\frac{mxy}{tmx+(1-t)y}\in D.$$ 
\end{definition}

In this case \.{I}. \.{I}scan in \cite{Iscan2016} generalized the harmonically convex function definition introduced in \cite{Iscan2014} to harmonically $(\alpha,m)$-convex function:

\begin{definition}(see \cite{Iscan2016})
	Let $f: D\subset (0,\infty) \rightarrow \R$ a function. Then,  $f$ is said to be harmonically $(\alpha,m)$-convex function if for all $\alpha, t\in [0,1]$, $m\in (0,1]$ and $x,y\in D$, we have:
	\begin{equation}\label{1.1}
	f\left(\frac{mxy}{tmx+(1-t)y}\right)\leq t^{\alpha}f(y)+m(1-t)^{\alpha}f(x).
	\end{equation}

	Note that if we considered $\alpha=1$ in (\ref{1.1}), $f$ is said to be a harmonically $m$-convex function and satisfies the following:
	
	$$f\left(\frac{mxy}{tmx+(1-t)y}\right)\leq tf(y)+m(1-t)f(x).$$
\end{definition}	

	For this kind of functions in \cite{Iscan2016} obtain the following result:

\begin{theorem}(see \cite{Iscan2016})
	Let $f :(0,\infty)\rightarrow \R$ be a harmonically $m$-convex function with
	$m\in (0,1]$. If $0 <a<b<\infty$ and $f\in L[a,b]$, then one has the inequality:
	$$\frac{ab}{b-a}\int_{a}^{b}\frac{f(x)}{x^{2}}dx\leq min\left\{\frac{f(a)+f\left\{\frac{b}{m}\right\}}{2},\frac{f(b)+f\left\{\frac{a}{m}\right\}}{2}\right\}.$$
\end{theorem}

In the other hand, G. Santana \textit{et al.} \cite{Santana2018} in 2018,  introduced the definition of harmonically convex set-valued functions, extending the definition given by \.I. \.{I}scan for real functions (see \cite{Iscan2014}).

\begin{definition}(see \cite{Santana2018})
	Let $X$ and $Y$ linear spaces, $D$ a harmonically convex subset of $X$ and $F: D\subset X\rightarrow n(Y)$ a set-valued function. Then $F$ is said harmonically convex function if for all $x,y\in D$ and $t\in [0,1]$, we have:
	$$tF(y)+(1-t)F(x)\subseteq F\left(\frac{xy}{tx+(1-t)y}\right)$$  
\end{definition}

\begin{remark}
	Throughout this paper $n(Y)$ will denote the family of nonempty subsets of $Y$.
\end{remark}

For this kind of functions they obtain many results as algebraic properties, Hermite-Hadamard and Fejer type inequalities and Bernstein-Doetsch type result.\\

To prove certain algebraic properties of the results in this research, we use two definitions established by T. Lara \textit{et al.} in 2014 \cite{Lara2019}.

\begin{definition}(see \cite{Lara2019})
	Let $F_{1},F_{2}: D\subseteq X\rightarrow n(Y)$ be two set-valued functions (or multifunctions) then:
	\begin{itemize}
		\item The union of $F_{1}$ and $F_{2}$ is a set-valued function $F_{1}\cup F_{2}: D\rightarrow n(Y)$ given by $(F_{1}\cup F_{2})(x)= F_{1}(x)\cup F_{2}(x)$ for each $x\in D$.
		\item The sum of $F_{1}$ and  $F_{2}$ is the function $F_{1}+F_{2}: D\rightarrow n(Y)$ defined in its usual form $(F_{1}+F_{2})(x)= F_{1}(x)+F_{2}(x)$ for each $x\in D$.
	\end{itemize} 
\end{definition}

\begin{definition}(see \cite{Lara2019})
	Let $X,Y,Z$ be linear spaces and $D$ be a subset of  $X$. Then:
	\begin{itemize}
		\item If $F_{1}: D\rightarrow n(Y)$ and $F_{2}: D\rightarrow n(Z)$, then the cartesian product function
		of $F_{1}$ and $F_{2}$ is the set-valued function $F_{1}\times F_{2}: D\rightarrow n(Y)\times n(Z)$ given by $(F_{1}\times F_{2})(x)= F_{1}(x)\times F_{2}(x)$ for each $x\in D$.
		
		\item If $F_{1}: D\rightarrow n(Y)$ and $F_{2}: n(Y)\rightarrow n(Z)$, then the composition function of $F_{1}$ and $F_{2}$ is the set-valued function $F_{2}\circ F_{1}: D\rightarrow n(Z)$ given by 	$$\displaystyle{(F_{2}\circ F_{1})(x)= F_{2}(F_{1}(x))= \bigcup_{y\in F_{1}(x)}F_{2}(y)},$$
		for each $x\in D$. 
	\end{itemize}
\end{definition}

Following the idea establish in \cite{Santana2018}, in this paper we extend that definition and introduce a new concept of convexity, combining the definitions of harmonically $m$-convex functions and set-valued functions. Then we define the following:

\begin{definition}
	Let $X$ and $Y$ be linear spaces, $D$ a harmonically $m$-convex subset of $X$ and $F:D\subset X\rightarrow n(Y)$ a set-valued function. It said that $F$ is said harmonically $m$-convex function if for all $x,y\in D$, $t\in [0,1]$, and $m\in (0,1]$, we have:
	
	\begin{equation}\label{FCVAmC_3}
	tF(y)+m(1-t)F(x)\subseteq F\left(\frac{mxy}{tmx+(1-t)y}\right).
	\end{equation}
\end{definition}

We have some examples of this kind of function.

\begin{example}
	Let $f_{1}, (-f_{2}): [a,b]\subset \R \rightarrow \R$ be harmonically $m$-convex functions with $f_{1}(x)\leq f_{2}(x)$ for all $x\in [a,b]$. Then, the set-valued function $F:D\subset X\rightarrow n(Y)$ defined by $F(x)= [f_{1}(x),f_{2}(x)]$ is harmonically $m$-convex.  
\end{example}

\noindent In fact, since $f_{1}$ and $-f_{2}$ are harmonically $m$-convex functions, then for all $x,y\in D$, $t\in [0,1]$ and $m\in (0,1]$, we have:
$$f_{1}\left(\frac{mxy}{tmx+(1-t)y}\right)\leq tf_{1}(y)+m(1-t)f_{1}(x),$$
and
$$-f_{2}\left(\frac{mxy}{tmx+(1-t)y}\right)\leq -(tf_{2}(y)+m(1-t)f_{2}(x)),$$
if we multiply $(-1)$ to both sides of the last inequality, we have:
$$f_{2}\left(\frac{mxy}{tmx+(1-t)y}\right)\geq tf_{2}(y)+m(1-t)f_{2}(x),$$ 

\noindent Then, 
\begin{equation}\label{1.2}
	[tf_{1}(y)+m(1-t)f_{1}(x), tf_{2}(y)+m(1-t)f_{2}(x)]\subset \left[f_{1}\left(\frac{mxy}{tmx+(1-t)y}\right), f_{2}\left(\frac{mxy}{tmx+(1-t)y}\right)\right]
\end{equation}

\noindent from (\ref{1.2}) we obtain with a elementary calculus that
$$tF(y)+m(1-t)F(x)\subset F\left(\frac{mxy}{tmx+(1-t)y}\right).$$
\begin{example}
	Let $H\in \R^{3}$ a harmonically $m$-convex subset, $F: \R \rightarrow n(\R)$ a set-valued function defined by $F(x)=f(x)H$, where $f(x)=x^{2}$ is a harmonically $m$-convex function.
\end{example}
\noindent Then, since $f(x)=x^{2}$ is a harmonically $m$-convex function by definition, we have for all $x,y\in D$, $t\in[0,1]$ and $m\in (0,1]$ that

$$\left(\frac{mxy}{tmx+(1-t)y}\right)^{2}\leq t(y)^{2}+m(1-t)(x)^{2},$$
using the harmonically $m$-convex properties of $H$ we have to
$$(t(y)^{2}+m(1-t)(x)^{2})H\subseteq \left(\frac{mxy}{tmx+(1-t)y}\right)^{2}H,$$
then,
\begin{eqnarray*}
  (t(y)^{2}+m(1-t)(x)^{2})H &=& t(y)^{2}H+m(1-t)(x)^{2}H\\
  &=& tF(y)+m(1-t)F(x)\\
  &\subseteq& F\left(\frac{mxy}{tmx+(1-t)y}\right).
\end{eqnarray*} 
Thus $F$ is a harmonically $m$-convex set-valued function.

\section{Main results}

The results obtained in this paper are based on the developments and ideas of \.I. Iscan \cite{Iscan2016} and T. Lara \textit{et al.} in \cite{Lara2019}. The following proposition establishes a property over harmonic $m$-convex set.

\begin{proposition}
	Let harmonic $m$-convex ($m\neq 1$) subset $D$ of $X$ is said to be \textit{starshaped} if, for all $x$ in $D$ and all $t$ in the interval $(0, 1]$, the point $tx$ also belongs to $D$. That is:
	$$tD\subseteq D.$$ 
\end{proposition} 

\begin{proof}
Let $D$ be a harmonically $m$-convex subset of $X$. If $D$ is an empty set, there is nothing to prove. If, on the contrary, we consider $D$ a nonempty set, let $x\in D$ then the point $x= \frac{mab}{tma+(1-t)b}\in D$ for everything $a,b\in D$ and $t\in[0,1]$.
Thus, $[m,1]x=\{rx: m\leq r \leq 1\}\subset D$, in particular $mx\in D$. If $m=0$, then $[0,1]x\in D$, we got the desired result.

In the case $m>0$, we similarly repeat the previous argument for $mx$ (instead of $x$), in this case we have to $[m^2,m]x=[m,1]mx\subseteq D$.

Inductively, we have that $[m^n,m^{n-1}]x\subseteq D$ for all $n\in \mathbb{N}$. Therefore $(0,1]x= \bigcup_{n=1}^{\infty}[m^n, m^{n-1}]x\subseteq D$. Thus $D$ satisfies $tD\subseteq D$ for $t\in (0,1]$.
\end{proof}


\noindent 	For harmonically $m$-convex set-valued functions  we obtain the following results:

\begin{proposition}
	Let  $F_{1},F_{2}:D\rightarrow n(Y)$ be harmonically $m$-convex set-valued functions with $F_{1}(x)\subseteq F_{2}(x)$ (or $F_{2}(x)\subseteq F_{1}(x)$) for each $x\in D$. Then the union of $F_{1}$ and $F_{2}$ ($F_{1}\cup F_{2}$) is a harmonically $m$-convex set-valued function.
\end{proposition}

\begin{proof} Let $F_{1},F_{2}$ be harmonically $m$-convex set-valued function with $x,y\in D$, $t\in [0,1]$ and $m\in (0,1]$. Let's assume $F_{1}(x)\subseteq F_{2}(x)$ (in the case $F_{2}(x)\subseteq F_{1}(x)$ is analogous) for each $x\in D$, then:

\begin{eqnarray*}
	t(F_{1}\cup F_{2})(y)&+&m(1-t)(F_{1}\cup F_{2})(x) \\
	&&= t(F_{1}(y)\cup F_{2}(y))+ m(1-t)(F_{1}(x)\cup F_{2}(x)\\
	&&= tF_{2}(y)+m(1-t)F_{2}(x)\\
	&&\subseteq F_{2}\left(\frac{mxy}{tmx+(1-t)y}\right)\\
	&&= F_{1}\left(\frac{mxy}{tmx+(1-t)y}\right)\cup F_{2}\left(\frac{mxy}{tmx+(1-t)y}\right)\\
	&&= (F_{1}\cup F_{2})\left(\frac{mxy}{tmx+(1-t)y}\right).
\end{eqnarray*} 
\end{proof}

\begin{proposition}\label{Prop_2_8}
	If $F: D\subset X \rightarrow n(Y)$ is a harmonically $m$-convex set-valued function, then the image of $F$ of any harmonically $m$-convex subset of $D$ is a $m$-convex set of $Y$.
\end{proposition}

\begin{proof}
Let $A$ be a harmonically $m$-convex subset of $D\subset X$ and $a,b\in F(A)=\cup_{z\in A}F(z)$. Then $a\in F(x)$ and $b\in F(y)$ for some $x,y \in A$. Thus, for all $t\in [0,1]$ and $m\in (0,1]$, we have to:

$$tb+m(1-t)a\in tF(y)+m(1-t)F(x)\subseteq F\left(\frac{mxy}{tmx+(1-t)y}\right).$$

Since $A$ is harmonically $m$-convex set, we have $\frac{mxy}{tmx+(1-t)y}\in A$ and
$tb+m(1-t)a\in F(A)$ for all $t \in [0,1]$. Which implies that $F(A)$ is $m$-convex set of $Y$.
\end{proof}

\begin{corollary}
	If $F: D\subseteq X\rightarrow n(Y)$ is a harmonically $m$-convex set-valued function, then the range of $F$ is $m$-convex set of $Y$.	 
\end{corollary}	 

\begin{proof}	
If we consider $A=D$ in Proposition \ref{Prop_2_8} we get
that $Rang(F)=F(D)$.
\end{proof}

\begin{proposition}
	A set-valued function $F:D\rightarrow n(Y)$ is harmonically $m$-convex, if and only if,
	\begin{equation}\label{FCVAmC_4}
	tF(B)+m(1-t)F(A)\subseteq F\left(\frac{mAB}{tmA+(1-t)B}\right),
	\end{equation} 
	for each $A,B\subseteq D$, $t\in [0,1]$ and $m\in (0,1]$.
\end{proposition}

\begin{proof}
$(\Rightarrow)$ 
Let $A,B$ be arbitrary subsets of $D$, $t\in [0,1]$ and $m\in (0,1]$.
Let $x\in t(F(B)=\cup_{b\in B}F(b))+m(1-t)F((A)=\cup_{a\in A}F(a)$, that is to say $x\in tF(b)+m(1-t)F(a)$ for some $a\in A$ and $b\in B$. Since $F$ is harmonically $m$-convex and $a,b\in D$, it follows that:

$$tF(b)+m(1-t)F(a)\subseteq F\left(\frac{mab}{tma+(1-t)b}\right),$$

\noindent moreover, $\frac{mab}{tma+(1-t)b}\in \frac{mAB}{tmA+(1-t)B}$ and, in consequence,

$$F\left(\frac{mab}{tma+(1-t)b}\right)\subset F\left(\frac{mAB}{tmA+(1-t)B}\right).$$ 

\noindent Therefore, $x\in F\left(\frac{mAB}{tmA+(1-t)B}\right)$.

$(\Leftarrow)$
For $x,y\in D$ $t\in [0,1]$ and $m\in (0,1]$, the result is obtained by replacing $A=\{x\}$
and $B=\{y\}$ in (\ref{FCVAmC_4}) and we obtain the desired result
$$tF(\{y\})+m(1-t)F(\{x\})\subseteq F\left(\frac{m\{x\}\{y\}}{tm\{x\}+(1-t)\{y\}}\right).$$

\noindent Then, $F$ is a harmonically $m$-convex set-valued function.
\end{proof}
   
In the following, consider that for nonempty linear space subsets $A,B,C,D$ and $\alpha$ a scalar, the following properties are true:

\begin{itemize}
	\item $\alpha(A\times B)= \alpha A\times \alpha B$,
	\item $(A\times C)+ (B\times D)= (A+B)\times (C+D)$,
	\item Si $A\subseteq B$ y $C\subseteq D$ then $A\times C\subseteq B\times D$.
\end{itemize} 

\begin{proposition}
	Let $F_{1}: D\rightarrow n(Y)$ and $F_{2}: D\rightarrow n(Z)$ harmonically $m$-convex set-valued functions. Then the cartesian product
	$F_{1}\times F_{2}$ is a harmonically $m$-convex set-valued function.
\end{proposition}

\begin{proof}
Let $x,y\in D$ and $t\in [0,1]$, then:

\begin{eqnarray*}
	t(F_{1}\times F_{2})(y)&+& m(1-t)(F_{1}\times F_{2})(x)\\
	&&= [tF_{1}(y)\times tF_{2}(y)]+[m(1-t)F_{1}(x)\times m(1-t)F_{2}(x)]\\
	&&=(tF_{1}(y)+m(t-1)F_{1}(x))\times (tF_{2}(y)+m(t-1)F_{2}(x))\\
	&&\subseteq F_{1}\left(\frac{mxy}{tmx+(1-t)y}\right)\times F_{2}\left(\frac{mxy}{tmx+(1-t)y}\right)\\
	&&=(F_{1}\times F_{2})\left(\frac{mxy}{tmx+(1-t)y}\right).
\end{eqnarray*}

\end{proof}

The following proposition establishes that the harmonically $m$-convex set-valued functions are closed under the sum and the product by a scalar.

\begin{proposition}
	Let $X,Y$ be two linear spaces. If $D$ is a harmonically $m$-convex subset of $X$ and $F,G: D\subset X\rightarrow n(Y)$ two harmonically $m$-convex set-valued functions. Then $\lambda F+ G$ is a harmonically $m$-convex set-valued function, for all $\lambda$.
\end{proposition} 

\begin{proof}
Let $x,y\in D\subset X$, $t\in [0,1]$ and $m\in (0,1]$. Since $F$ and $G$
are harmonically $m$-convex set-valued functions, we have:

$$tF(y)+m(1-t)F(x)\subseteq F\left(\frac{mxy}{tmx+(1-t)y}\right),$$
and,
$$tG(y)+m(1-t)G(x)\subseteq G\left(\frac{mxy}{tmx+(1-t)y}\right).$$

\noindent Thus,

\begin{eqnarray*}
	t(\lambda F+ G)(y)&+& m(1-t)(\lambda F+ G)(x)\\
	&&= [t(\lambda F(y))+m(1-t)(\lambda F(x))]+[tG(y)+m(1-t)G(x)]\\
	&& \subseteq\lambda F\left(\frac{mxy}{tmx+(1-t)y}\right)+ G\left(\frac{mxy}{tmx+(1-t)y}\right)\\
	&&= (\lambda F+G)\left(\frac{mxy}{tmx+(1-t)y}\right).
\end{eqnarray*}
\end{proof}

\begin{proposition}
	Let $x,\, Y$ be two linear spaces. If $D$ is a harmonically $m$-convex subset of $X$ and $F,G: D\subset X\rightarrow n(Y)$ two harmonically $m$-convex set-valued functions then $(F\cdot G)(x)$ is also harmonically $m$-convex set-valued function.
\end{proposition}

\begin{proof}
	First, from \cite{Santana2018}, we have that.
	
	\[F(x_1)G(x_2)+F(x_2)G(x_1)\subset F(x_1)G(x_1)+F(x_2)G(x_2).\]
	
	Then, for $t\in[0,1]$ and $x_1,\,x_2\in D$:
	
	\small{\begin{eqnarray*}
		(F\cdot G)\left(\frac{mx_1x_2}{mtx_1+(1-t)x_2}\right)&=& F\left(\frac{mx_1x_2}{mtx_1+(1-t)x_2}\right)\cdot G\left(\frac{mx_1x_2}{mtx_1+(1-t)x_2}\right)\\
		&\supseteq&[tF(x_1)+m(1-t)F(x_2)][tG(x_1)+m(1-t)G(x_2)]\\
		&=& t^2F(x_1)G(x_1)+mt(1-t)F(x_1)G(x_2)+mt(1-t)F(x_2)G(x_1)\\
		&&+m^2(1-t)^2F(x_2)G(x_2)\\
		&=& t^2F(x_1)G(x_1)+mt(1-t)[F(x_1)G(x_2)+F(x_2)G(x_1)]\\
		&& m^2(1-t)^2F(x_2)G(x_2)\\
		&\supseteq& t^2F(x_1)G(x_1)+mt(1-t)[F(x_1)G(x_1)+F(x_2)G(x_2)]\\
		&&+m^2(1-t)^2F(x_2)G(x_2)\\
		&=& t[t+m(1-t)]F(x_1)G(x_1)+[m(1-t)[t+m(1-t)]F(x_2)G(x_2)]\\
		&\supseteq& tF(x_1)G(x_1)+m(1-t)F(x_2)G(x_2).
	\end{eqnarray*}}
	
	This shows that the product of two harmonically $m$-convex set-valued functions is again harmonically $m$-convex set-valued function.
\end{proof}


The following result follows the idea of \.I. \.{I}scan in \cite{Iscan2016}. To integrate set-valued functions we use the definition given by R. J. Aumann, and  if a function satisfies the requirements of being integrable under this integral definition given, we say that a set-valued function $F$ is Aumann integrable under a certain domain (see \cite{Aumann1965}).

\begin{theorem}
	Let  $X,Y$ linear spaces, $D$ be a harmonically
	$m$-convex subset of $X$ and $F:D\subset X\rightarrow n(Y)$ a harmonically $m$-convex set-valued Aumann integrable function, then
	
	$$\min\left(\inf\left(\frac{F(a)+mF\left(\frac{b}{m}\right)}{2}\right),\inf\left(\frac{mF\left(\frac{a}{m}\right)+F(b)}{2}\right)\right)\subseteq \frac{ab}{b-a}\int_{a}^{b}\frac{F(x)}{x^2}dx.$$ 
\end{theorem}

\begin{proof}
Let $F:D\subset X\rightarrow n(Y)$ be a harmonically $m$-convex set-valued function, for every $x,y\in D$ we have to

$$tF(y)+m(1-t)F\left(\frac{x}{m}\right)\subset F\left(\frac{xy}{tx+(1-t)y}\right)=F\left(\frac{m\frac{x}{m}y}{tm\frac{x}{m}+(1-t)y}\right),$$

\noindent Then, we have the following:
\begin{equation}\label{32}
tF(b)+m(1-t)F\left(\frac{a}{m}\right)\subset F\left(\frac{ab}{ta+(1-t)b}\right),
\end{equation}

\noindent and
$$tF(a)+m(1-t)F\left(\frac{b}{m}\right)\subset F\left(\frac{ab}{tb+(1-t)a}\right),$$
\noindent for every  $t\in [0,1]$, $m\in (0,1]$ and $a,b\in D$. Integrating both sides of (\ref{32}) on $[0,1]$ with respect to $t$, we get that 

\begin{equation}\label{es}
  \int_{0}^{1}tF(b)+m(1-t)F\left(\frac{a}{m}\right)dt\subset \int_{0}^{1}F\left(\frac{ab}{ta+(1-t)b}\right)dt.
\end{equation}

\noindent Integrating the left side of (\ref{es}) we have:
$$\int_{0}^{1}tF(b)+m(1-t)F\left(\frac{a}{m}\right)dt= \frac{F(b)+mF\left(\frac{a}{m}\right)}{2}.$$

\noindent By the Aumann integral definition we get, that integral on the right hand of (\ref{es}) is defined as:

$$ \int_{0}^{1}F\left(\frac{ab}{ta+(1-t)b}\right)dt= \left\{ \int_{0}^{1}f\left(\frac{ab}{ta+(1-t)b}\right)dt: f(x)\in F(x)\wedge t\in [0,1]\right\}.$$

\noindent But,
$$\int_{0}^{1}f\left(\frac{ab}{ta+(1-t)b}\right)dt= \frac{ab}{b-a}\int_{a}^{b}\frac{f(x)}{x^2}dx,$$

\noindent then 
$$\left\{\frac{ab}{b-a}\int_{a}^{b}\frac{f(x)}{x^2}dx: f(x)\in F(x)\wedge x\in [a,b]\right\}= \frac{ab}{b-a}\int_{a}^{b}\frac{F(x)}{x^2}dx.$$

\noindent In consecuense:
$$\frac{F(b)+mF\left(\frac{a}{m}\right)}{2}\subseteq \frac{ab}{b-a}\int_{a}^{b}\frac{F(x)}{x^2}dx.$$

\noindent Similarly, we have that:
$$\frac{F(a)+mF\left(\frac{b}{m}\right)}{2}\subseteq \frac{ab}{b-a}\int_{a}^{b}\frac{F(x)}{x^2}dx,$$

\noindent so the required result is obtained. 
$$\min\left(\inf\left(\frac{F(a)+mF\left(\frac{b}{m}\right)}{2}\right),\inf\left(\frac{mF\left(\frac{a}{m}\right)+F(b)}{2}\right)\right)\subseteq \frac{ab}{b-a}\int_{a}^{b}\frac{F(x)}{x^2}dx.$$ 
\end{proof}

 \end{document}